\documentclass[11pt]{article}
\usepackage[margin=1in]{geometry} 
\usepackage{amssymb,amsfonts,amsmath,bbm,mathrsfs,stmaryrd,mathtools}
\usepackage{xcolor}
\usepackage{url}

\usepackage{extarrows}

\usepackage[shortlabels]{enumitem}
\usepackage{tensor}

\usepackage{xr}
\usepackage[T1]{fontenc}

\usepackage[colorlinks,
linkcolor=black!75!red,
citecolor=blue,
pdftitle={},
pdfproducer={pdfLaTeX},
pdfpagemode=None,
bookmarksopen=true,
bookmarksnumbered=true,
backref=page]{hyperref}

\usepackage{tikz}
\usetikzlibrary{arrows,calc,decorations.pathreplacing,decorations.markings,intersections,shapes.geometric,through,fit,shapes.symbols,positioning,decorations.pathmorphing}

\usepackage{braket}

\usepackage[amsmath,thmmarks,hyperref]{ntheorem}
\usepackage{cleveref}

\newcommand\nthalias[1]{\AddToHook{env/#1/begin}{\crefalias{lemma}{#1}}}

\nthalias{definition}
\nthalias{example}
\nthalias{examples}
\nthalias{remark}
\nthalias{remarks}
\nthalias{convention}
\nthalias{notation}
\nthalias{construction}
\nthalias{theoremN}
\nthalias{propositionN}
\nthalias{corollaryN}
\nthalias{lemma}
\nthalias{proposition}
\nthalias{corollary}
\nthalias{theorem}
\nthalias{conjecture}
\nthalias{question}
\nthalias{assumption}

\creflabelformat{enumi}{#2#1#3}

\crefname{section}{Section}{Sections}
\crefformat{section}{#2Section~#1#3} 
\Crefformat{section}{#2Section~#1#3} 

\crefname{subsection}{\S}{\S\S}
\AtBeginDocument{%
  \crefformat{subsection}{#2\S#1#3}%
  \Crefformat{subsection}{#2\S#1#3}%
}

\crefname{subsubsection}{\S}{\S\S}
\AtBeginDocument{%
  \crefformat{subsubsection}{#2\S#1#3}%
  \Crefformat{subsubsection}{#2\S#1#3}%
}

\theoremstyle{plain}

\newtheorem{lemma}{Lemma}[section]
\newtheorem{proposition}[lemma]{Proposition}

\newtheorem{theorem}[lemma]{Theorem}

\theoremstyle{plain}
\theoremnumbering{Alph}

\theoremstyle{plain}
\theorembodyfont{\upshape}
\theoremsymbol{\ensuremath{\blacklozenge}}

\newtheorem{definition}[lemma]{Definition}
\newtheorem{example}[lemma]{Example}

\newtheorem{remark}[lemma]{Remark}

\crefname{definition}{definition}{definitions}
\crefformat{definition}{#2definition~#1#3} 
\Crefformat{definition}{#2Definition~#1#3} 

\crefname{ex}{example}{examples}
\crefformat{example}{#2example~#1#3} 
\Crefformat{example}{#2Example~#1#3} 

\crefname{exs}{example}{examples}
\crefformat{examples}{#2example~#1#3} 
\Crefformat{examples}{#2Example~#1#3} 

\crefname{remark}{remark}{remarks}
\crefformat{remark}{#2remark~#1#3} 
\Crefformat{remark}{#2Remark~#1#3} 

\crefname{remarks}{remark}{remarks}
\crefformat{remarks}{#2remark~#1#3} 
\Crefformat{remarks}{#2Remark~#1#3} 

\crefname{convention}{convention}{conventions}
\crefformat{convention}{#2convention~#1#3} 
\Crefformat{convention}{#2Convention~#1#3} 

\crefname{notation}{notation}{notations}
\crefformat{notation}{#2notation~#1#3} 
\Crefformat{notation}{#2Notation~#1#3} 

\crefname{table}{table}{tables}
\crefformat{table}{#2table~#1#3} 
\Crefformat{table}{#2Table~#1#3}

\crefname{lemma}{lemma}{lemmas}
\crefformat{lemma}{#2lemma~#1#3} 
\Crefformat{lemma}{#2Lemma~#1#3} 

\crefname{proposition}{proposition}{propositions}
\crefformat{proposition}{#2proposition~#1#3} 
\Crefformat{proposition}{#2Proposition~#1#3} 

\crefname{propositionN}{proposition}{propositions}
\crefformat{propositionN}{#2proposition~#1#3} 
\Crefformat{propositionN}{#2Proposition~#1#3} 

\crefname{corollary}{corollary}{corollaries}
\crefformat{corollary}{#2corollary~#1#3} 
\Crefformat{corollary}{#2Corollary~#1#3} 

\crefname{corollaryN}{corollary}{corollaries}
\crefformat{corollaryN}{#2corollary~#1#3} 
\Crefformat{corollaryN}{#2Corollary~#1#3} 

\crefname{theorem}{theorem}{theorems}
\crefformat{theorem}{#2theorem~#1#3} 
\Crefformat{theorem}{#2Theorem~#1#3} 

\crefname{theoremN}{theorem}{theorems}
\crefformat{theoremN}{#2theorem~#1#3} 
\Crefformat{theoremN}{#2Theorem~#1#3} 

\crefname{enumi}{}{}
\crefformat{enumi}{#2#1#3}
\Crefformat{enumi}{#2#1#3}

\crefname{assumption}{assumption}{Assumptions}
\crefformat{assumption}{#2assumption~#1#3} 
\Crefformat{assumption}{#2Assumption~#1#3} 

\crefname{construction}{construction}{Constructions}
\crefformat{construction}{#2construction~#1#3} 
\Crefformat{construction}{#2Construction~#1#3} 

\crefname{question}{question}{Questions}
\crefformat{question}{#2question~#1#3} 
\Crefformat{question}{#2Question~#1#3} 

\crefname{equation}{}{}
\crefformat{equation}{(#2#1#3)} 
\Crefformat{equation}{(#2#1#3)}

\numberwithin{equation}{section}

\theoremstyle{nonumberplain}
\theoremsymbol{\ensuremath{\blacksquare}}

\newtheorem{proof}{Proof}
\newcommand\pf[1]{\newtheorem{#1}{Proof of \Cref{#1}}}

\newcommand\cC{{\mathcal C}}

\newcommand\cM{{\mathcal M}}

\newcommand\cS{{\mathcal S}}

\DeclareMathOperator{\id}{id}

\title{Coassociative structures on self-injective algebras}
\author{Alexandru Chirvasitu}

\begin{document}

\date{}

\newcommand{\Addresses}{{
  \bigskip
  \footnotesize

  \textsc{Department of Mathematics, University at Buffalo, Buffalo,
    NY 14260-2900, USA}\par\nopagebreak \textit{E-mail address}:
  \texttt{achirvas@buffalo.edu}

}}

\maketitle

\begin{abstract}
  For general finite-dimensional self-injective algebra $A$ we construct a family of injective coassociative coproducts $A\to A\otimes A$, all $A$-bimodule morphisms. In particular such structures always exist, confirming a conjecture of Hernandez, Walton and Yadav. The coproducts are indexed by subsets of $\{1,\cdots,m(i)\}\times  \{1,\cdots,m(\nu^{-1}i)\}$, where $A\cong \mathrm{End}_{\Lambda}(M)$ is the general form of a self-injective algebra in terms of a basic Frobenius $\Lambda$, the $m(i)$, $1\le i\le n$ are the multiplicities of the indecomposable projective $\Lambda$-modules in $M$, and $\nu$ is the Nakayama permutation of $\Lambda$. We also characterize those among the coproducts introduced in this fashion, in terms this combinatorial data, which are counital. 
\end{abstract}


\vspace{.5cm}

\noindent{MSC 2020: 16D50; 16L60; 16T15; 18M05; 16U40; 16D40; 16D60; 16D20
  
}


\section*{Introduction}

The present note is motivated by questions posed in \cite{MR4834012} in the context of analyzing self-dual structures pertaining to \emph{Frobenius $\Bbbk$-algebras} ($\Bbbk$ being a field, here and throughout the paper). Recall (e.g. \cite[\S 27.2]{bw}) that one way to define the latter is as
\begin{itemize}[wide]
\item unital, associative $\Bbbk$-algebras;
\item and simultaneously counital, coassociative \emph{coalgebras} (\cite[\S 1.1]{bw}, \cite[Definition 1.1.3]{mntg}, etc.);

\item with the comultiplication $A\xrightarrow{\Delta}A\otimes A$ being an $A$-bimodule morphism. 
\end{itemize}
One checks immediately that the latter condition, given the other two, is equivalent to the multiplication $A\otimes A\to A$ being an $A$-bi\emph{co}module morphism instead (hence the opening sentence's ``self-duality'' ). Finite dimensionality is automatic \cite[\S 27.9]{bw} for Frobenius algebras, as is \emph{self-injectivity} \cite[Proposition 3.14]{MR1653294}: $A$ will be injective as both a left and right $A$-module.

It is in this context that \cite{MR4834012} analyzes how and to what extent some comultiplicative structure survives weakening the Frobenius condition. Specifically, \cite[Theorems 1.7 and 1.9]{MR4834012} construct non-counital (but one-to-one, so very much non-trivial) comultiplications on various classes of self-injective (finite-dimensional) algebras:
\begin{itemize}[wide]
\item certain self-injective algebras constructed based on cyclic quivers and referred to in \cite[Definition 1.6]{MR4834012} as \emph{NSY algebras} (also the \emph{bound quiver algebras} \cite[pre Lemma I.1.5]{sy-bk1} featuring in \cite[Theorem IV.6.15]{sy-bk1}); 

\item and respectively the \emph{weak Hopf algebras} of \cite[Definition 2.1]{MR1726707}. 
\end{itemize}
In both cases the coproducts are attached uniformly to individual members of the class of algebras in question in such a manner that it will be counital precisely when the algebra is Frobenius (as opposed to only self-injective). \cite[Conjecture 1.10]{MR4834012} (which affords some flexibility of interpretation) voices an expectation that such structures obtain for arbitrary finite-dimensional self-injective algebras. We prove that conjecture in the following form. 

\begin{theorem}\label{th:ncounit}
  Every self-injective finite-dimensional algebra $A$ admits an injective coassociative $A$-bimodule morphism $\Delta:A\to A\otimes A$.
\end{theorem}

A more elaborate version of the statement (\Cref{th:ncounit-bis}) will in fact produce, for every self-injective finite-dimensional $A$, a family of comultiplications indexed by combinatorial data consisting of numerical structural invariants of $A$. \Cref{pr:whencounit} amplifies the construction by further characterizing those among the constructed coproducts which are counital.

To expand, we remind the reader \cite[\S 1]{sy-gen} that every self-injective algebra as in the statement of \Cref{th:ncounit} is of the form
\begin{equation}\label{eq:amis}
  A\cong \Lambda(m(1),\cdots,m(n)) := \mathrm{End}_{\Lambda}(M_1\oplus\cdots\oplus M_n)
\end{equation}
where
\begin{itemize}[wide]
\item $\Lambda$ is a self-injective {\it basic} (hence Frobenius) algebra. Being basic means the decomposition
  \begin{equation}\label{eq:lbd}
    \Lambda\cong P_1\oplus \cdots\oplus P_n
  \end{equation}
  in $\cM_{\Lambda}$ into indecomposable projectives has mutually non-isomorphic summands: $P_i\not\cong P_j$ for $i\ne j$.
\item $M_i\cong P_i^{m(i)}$ is a direct sum of copies of $P_i$.
\end{itemize}
The comultiplications provided by \Cref{th:ncounit-bis} are indexed by subsets
\begin{equation*}
  \cS(i)\subseteq \{1,\cdots,m(i)\}\times  \{1,\cdots,m(\nu^{-1}i)\}
  ,\quad
  1\le i\le n
\end{equation*}
with $\nu$ denoting the \emph{Nakayama permutation} of the Frobenius algebra $\Lambda$ (inverse of the permutation denoted by the same symbol \cite[post Theorem IV.6.1]{sy-bk1}).

\subsection*{Acknowledgments}

I am grateful for helpful comments, suggestions and literature pointers from C. Walton and H. Yadav. 

\section{Preliminaries}\label{se.prel}

Algebras, unless specified otherwise, are unital, associative, and finite-dimensional over a fixed field $\Bbbk$. Algebra subscripts indicate module sidedness: $\mathrm{Hom}_A(-,-)$ means right-$A$-module morphisms, ${}_A\mathrm{Hom}(-,-)$ left-module morphisms, etc.

Idempotent elements
\begin{equation*}
  e^2=e\in A
\end{equation*}
will feature prominently below, so we remind the reader (e.g. \cite[\S I.5, discussion preceding Lemma 5.6]{sy-bk1}) that
\begin{itemize}[wide]
\item idempotents $e,f\in A$ are {\it orthogonal} if $ef=fe=0$;
\item and an idempotent is {\it primitive} if it cannot be decomposed as a sum of two non-zero orthogonal idempotents.
\end{itemize}

We make frequent (and implicit) use of the identifications \cite[Proposition 21.6]{lam}
\begin{equation*}
  \mathrm{Hom}_A(fA,eA)\cong eAf\cong {}_A\mathrm{Hom}(Ae,Af):
\end{equation*}
in the first case $eAf$ left-multiplies $fA$ into $eA$, and similarly in the other case (via right multiplication instead).

\cite[Lemma 2.12]{MR4834012} notes that for an algebra $A$ (assumed unital), specifying a coassociative $A$-bimodule morphism
\begin{equation}\label{eq:delta}
  \Delta:A\to A\otimes A
\end{equation}
is equivalent to selecting an element $t:=\sum a_i\otimes b_i$ in the bimodule $A$, {\it invariant} in the sense that $at=ta$ for all $a\in A$; that element will then simply be $\Delta(1)$. 

In general, for an $A$-bimodule $M$, a bimodule morphism $A\to M$ is again simply an invariant element $m\in M$, i.e. one belonging to $M^{(A)}$ of \cite[\S 10.4]{pierce}:
\begin{equation*}
  M^{(A)}:=\{m\in M\ |\ am=ma,\ \forall a\in A\}. 
\end{equation*}
As before, that element is the image of $1$ through the desired bimodule map $A\to M$. \cite[Lemma 2.12]{MR4834012} then notes that given such an element in $A\otimes A^{(A)}$, the resulting map \Cref{eq:delta} is automatically coassociative. This last observation is in fact a particular instance of the simple remark noted in \Cref{le:autoassoc}.

We will refer to {\it monoidal categories} \cite[\S VII.1]{mcl} $(\cC,\otimes,{\bf 1})$, typically denoting by
\begin{equation*}
  \alpha:(-\otimes -)\otimes - \cong -\otimes (-\otimes -)
\end{equation*}
the {\it associator} and by
\begin{equation*}
  \lambda: ({\bf 1}\otimes -)\cong \id
  \quad\text{and}\quad
  \rho: (-\otimes {\bf 1})\cong \id
\end{equation*}
the natural isomorphisms pertaining to unitality.

\begin{lemma}\label{le:autoassoc}
  Let $(\cC,\otimes,{\bf 1},\alpha,\lambda,\rho)$ be a monoidal category, $c\in \cC$ an object and $\varphi:{\bf 1}\to c$ a morphism.

  The two compositions
  \begin{equation*}
    {\bf 1}
    \xrightarrow{\quad\varphi\quad}
    c
    \xrightarrow{\quad\lambda^{-1}\quad}
    {\bf 1}\otimes c
    \xrightarrow{\quad\varphi\otimes\id\quad}c\otimes c
  \end{equation*}
  and
  \begin{equation*}
    {\bf 1}
    \xrightarrow{\quad\varphi\quad}
    c
    \xrightarrow{\quad\rho^{-1}\quad}
    c\otimes {\bf 1}
    \xrightarrow{\quad\id\otimes\varphi\quad}c\otimes c
  \end{equation*}
  are equal. 
\end{lemma}
\begin{proof}
  That both are equal to 
  \begin{equation*}
    {\bf 1}
    \xrightarrow{\quad\lambda=\rho\quad}
    {\bf 1}\otimes {\bf 1}
    \xrightarrow{\quad\varphi\otimes\varphi\quad}
    c\otimes c
  \end{equation*}
  follows from the functoriality of `$\otimes$', the naturality of $\lambda$ and $\rho$, and the fact that
  \begin{equation*}
    \lambda=\rho:{\bf 1}\to {\bf 1}\otimes{\bf 1}. 
  \end{equation*}
  (by definition: \cite[\S VII.1, equation (8)]{mcl}).
\end{proof}

To recover \cite[Lemma 2.12]{MR4834012} from \Cref{le:autoassoc} take
\begin{itemize}[wide]
\item $(\cC,\otimes,{\bf 1})$ to be the category ${}_A\cM_A$ of bimodules, with $\otimes_A$ as the tensor product and $A$ as the monoidal unit;
\item and the object $c\in \cC$ to be $A\otimes A$, noting that
  \begin{equation*}
    c\otimes c = (A\otimes A)\otimes_A (A\otimes A)\cong A\otimes A\otimes A.
  \end{equation*}
\end{itemize}
The conclusion of \Cref{le:autoassoc} is then precisely the coassociativity of $\Delta$.


\section{Non-counital comultiplications}\label{se:ncounit}

Some of the notation and language surrounding \Cref{eq:amis} will be useful in proving \Cref{th:ncounit} and expanding the discussion beyond that result. Following \cite[\S IV.6]{sy-bk1}:

\begin{definition}\label{def:candec}
  Let $A$ be a finite-dimensional $\Bbbk$-algebra. A {\it canonical decomposition} for $A$ is a decomposition
  \begin{equation}\label{eq:candec}
    1=\sum_{i=1}^n\sum_{j=1}^{m(i)} e_{ij}
  \end{equation}
  of the unit $1=1_A$ as a sum of primitive pairwise-orthogonal idempotents so that
  \begin{itemize}[wide]
  \item for each $i$, all $e_{ij}A$ are isomorphic to a single (indecomposable) projective right $A$-module $P_i$;
  \item while $e_{ij}A\not\cong e_{i'j'}A$ as soon as $i\ne i'$.
  \end{itemize}
  Given a canonical decomposition, we also write
  \begin{equation*}
    e_i:=\sum_{j=1}^{m(i)}e_{ij}. 
  \end{equation*}
  for the idempotent collecting all copies of $P_i$.
\end{definition}

Some useful notation:
\begin{itemize}[wide]
\item For indices $1\le i,j\le n$ numbering the indecomposable $\Lambda$-projectives we write
  \begin{equation*}
    \Lambda_{j\leftarrow i}\subset \Lambda\cong \mathrm{End}_{\Lambda}(P_1\oplus\cdots\oplus P_n)
  \end{equation*}
  for the $\Lambda$-morphisms $P_i\to P_j$, extended by zero elsewhere on \Cref{eq:lbd}.
\item We will fix and number copies $P_{is}$, $1\le s\le m(i)$ whose direct sum (ranging over all $i$ and $s$) is precisely
  \begin{equation*}
    M=M_1\oplus\cdots\oplus M_n = P_1^{m(1)}\oplus\cdots\oplus P_n^{m(n)}. 
  \end{equation*}
\item Consider a morphism $\varphi:P_i\to P_j$. We write $\varphi^{t\leftarrow s}$ for the corresponding morphism $P_{is}\to P_{jt}$, obtained from $\varphi$ via the identifications
  \begin{equation*}
    P_i\cong P_{is},\ P_j\cong P_{jt}.
  \end{equation*}
\item In particular, we have identifications $\id_i^{t\leftarrow s}:P_{is}\to P_{it}$.
\item Morphisms $P_{is}\to P_{jt}$ (such as the $\varphi^{t\leftarrow s}$ above) are also regarded as endomorphisms of $M$ (and hence elements of $A$), in the obvious fashion: extend them by zero on all summands
  \begin{equation*}
    P_{jt},\quad (j,t)\ne (i,s).
  \end{equation*}
\item Having identified $P_i\cong P_{i1}$, we regard $\Lambda$ as a subspace (non-unital subalgebra) of $A$, again by extending endomorphisms by zero.
\item Extending the notation $\Lambda_{j\leftarrow i}\subset \Lambda$, we write
  \begin{equation*}
    A_{j\leftarrow i}^{t\leftarrow s}\subset A
  \end{equation*}
  for the image of
  \begin{equation}\label{eq:morts}
    \Lambda_{j\leftarrow i}
    \cong 
    \mathrm{Hom}_{\Lambda}(P_i,P_j)\ni
    \varphi\mapsto \varphi^{t\leftarrow s}
    \in A.
  \end{equation}
  and
  \begin{equation*}
    A_{j\leftarrow i}:=\bigoplus_{s,t}A^{t\leftarrow s}_{j\leftarrow i}.
  \end{equation*}
  We then have
  \begin{equation*}
    A
    =
    \bigoplus_{i,j}\bigoplus_{s=1}^{m(i)}\bigoplus_{t=1}^{m(j)} A_{j\leftarrow i}^{t\leftarrow s}
    =
    \bigoplus_{i,j} A_{j\leftarrow i}.
  \end{equation*}
\end{itemize}

\pf{th:ncounit}
\begin{th:ncounit}
  We prove the claim in stages: constructing what we prove to be a bimodule map $\Delta:A\to A\otimes A$ first, and then showing that that map is one-to-one.

  \begin{enumerate}[(1),wide]
  \item {\bf The construction.} Per \cite[Lemma 2.12]{MR4834012}, it is enough to find an element $x\in A\otimes A$ invariant for the $A$-bimodule structure:
    \begin{equation}\label{eq:inv}
      ax=xa,\ \forall a\in A. 
    \end{equation}
    Because $\Lambda$ is Frobenius, we already have an element
    \begin{equation*}
      y = \sum_{\ell} \varphi_\ell \otimes \psi_{\ell}\in \Lambda\otimes\Lambda\subseteq A\otimes A
    \end{equation*}
    satisfying \Cref{eq:inv} for $a\in \Lambda\subseteq A$. Decomposing the tensorands therein in a basis compatible with the decomposition
    \begin{equation*}
      \Lambda = \bigoplus_{i,j}\Lambda_{j\leftarrow i},
    \end{equation*}
    we can rewrite $y$ as
    \begin{equation}\label{eq:ysum}
      y = \sum_{i,i',j}\sum_{\bullet} \varphi_{\bullet, j\leftarrow i}\otimes \psi_{\bullet, i'\leftarrow j}
      \quad\text{with}\quad
      \varphi_{\bullet, j\leftarrow i}\in \Lambda_{j\leftarrow i},\ \psi_{\bullet, i'\leftarrow j}\in \Lambda_{i'\leftarrow j};
    \end{equation}
    the fact that the same index $j$ appears at both extremes follows from our condition $ay=ya$, $a\in \Lambda$.

    To obtain $x$, simply ``spread'' each summand \Cref{eq:ysum} of $y$ over all copies of $P_j$:
    \begin{equation*}
      x = \sum_{i,i',j} \sum_{t=1}^{m(j)} \sum_{\bullet} \varphi^{t\leftarrow 1}_{\bullet, j\leftarrow i}\otimes \psi^{1\leftarrow t}_{\bullet, i'\leftarrow j}.
    \end{equation*}
    To conclude, note that
    \begin{itemize}[wide]
    \item we have $ax=xa$ for $a\in \Lambda$ because the same condition held for $y$, and
      \begin{equation*}
        a\in \Lambda\Rightarrow ax=ay,\ xa=ya. 
      \end{equation*}
    \item we also have
      \begin{equation*}
        \id_j^{t\leftarrow s} x = x \id_j^{t\leftarrow s},
      \end{equation*}
      essentially by construction: both are equal to
      \begin{equation*}
        \sum_{i,i'}\sum_{\bullet}\varphi^{t\leftarrow 1}_{\bullet,j\leftarrow i}\otimes \psi^{1\leftarrow s}_{\bullet,i'\leftarrow j}
      \end{equation*}
    \end{itemize}
    The conclusion \Cref{eq:inv} follows, since $A$ is generated as an algebra by its (non-unital) subalgebra $\Lambda$ and the morphisms $\id_j^{t\leftarrow s}$.
    
  \item {\bf Injectivity.} The comultiplication on $\Lambda$ we started with, induced by $y$ via
    \begin{equation*}
      \Lambda\ni a\mapsto ay\in \Lambda\otimes\Lambda,
    \end{equation*}
    is counital. Because, as noted, for $a\in \Lambda$ we have
    \begin{equation*}
      \Delta(a) = ax = ay\in \Lambda\otimes\Lambda\subseteq A\otimes A,
    \end{equation*}
    it follows that $\Delta$ is one-to-one when restricted to
    \begin{equation*}
      \Lambda\cong \bigoplus_{i,j}A_{j\leftarrow i}^{1\leftarrow 1}.
    \end{equation*}
    Next, note that (by construction) the comultiplication $\Delta$ is compatible with the morphisms $(\bullet)^{t\leftarrow s}$ of \Cref{eq:morts} in the sense that
    \begin{equation}\label{eq:laaa}
      \begin{tikzpicture}[auto,baseline=(current  bounding  box.center)]
        \path[anchor=base] 
        (0,0) node (l) {$\Lambda_{j\leftarrow i}$}
        +(4,.5) node (u) {$A_{j\leftarrow i}^{t\leftarrow s}$}
        +(4,-.5) node (d) {$\bigoplus A_{j\leftarrow \bullet}^{1\leftarrow 1}\otimes A_{\bullet\leftarrow i}^{1\leftarrow 1}$}
        +(8,0) node (r) {$\bigoplus A_{j\leftarrow \bullet}^{t\leftarrow 1}\otimes A_{\bullet\leftarrow i}^{1\leftarrow s}$}
        ;
        \draw[->] (l) to[bend left=6] node[pos=.5,auto] {$\scriptstyle (\bullet)^{t\leftarrow s}$} (u);
        \draw[->] (u) to[bend left=6] node[pos=.5,auto] {$\scriptstyle \Delta$} (r);
        \draw[->] (l) to[bend right=6] node[pos=.5,auto,swap] {$\scriptstyle \Delta$} (d);
        \draw[->] (d) to[bend right=6] node[pos=.5,auto,swap] {$\scriptstyle (\bullet)^{t\leftarrow 1}\otimes (\bullet)^{1\leftarrow s}$} (r);
      \end{tikzpicture}
    \end{equation}
    commutes. Because the bottom left-hand arrow is injective so is the top right-hand one (the north-eastward arrows being isomorphisms). The conclusion follows from the fact that the spaces $A_{j\leftarrow i}^{t\leftarrow s}$ (at the top of \Cref{eq:laaa}) are complementary direct summands of $A$ as
    \begin{equation*}
      1\le i,j\le n\quad\text{and}\quad 1\le s\le m(i),\ 1\le t\le m(j)
    \end{equation*}
    vary, while the rightmost spaces in \Cref{eq:laaa} are, similarly, mutually independent summands of $A\otimes A$.
  \end{enumerate}
  This finishes the proof. 
\end{th:ncounit}

\subsection{Prior work}\label{subse:pre}

It might be instructive to spell out how the comultiplication provided by \Cref{th:ncounit} compares to, say, that of \cite[Theorem 2.14]{MR4834012}. In that paper's notation, the self-injective algebra being studied is
\begin{equation*}
  A=\Lambda(m_0,\cdots,m_{n-1})
\end{equation*}
for the basic Frobenius bound-quiver algebra $\Lambda=B_{n,\ell}$. The proof of \cite[Theorem 2.14]{MR4834012} describes the element $\Delta(1)$ that determines the comultiplication. In the particular case when $A=B_{n,\ell}$, i.e. all $m_i$ are 1, that element specializes (again, in that paper's notation) to
\begin{equation*}
  B_{n,\ell}\otimes B_{n,\ell}
  \ni
  \Delta(1) = \sum_{i=0}^{n-1}\sum_{k=0}^{\ell-1} X_{i,k}^{0,0}\otimes X_{i+k-\ell+1,\ell-1-k}^{0,0}.
\end{equation*}
Passing to arbitrary $m_i$ and applying the ``spreading'' procedure described in the proof of \Cref{th:ncounit}, we then obtain
\begin{equation}\label{eq:d1}
  A\otimes A
  \ni
  \Delta(1) = \sum_{i=0}^{n-1}\sum_{r_i=0}^{m_i-1}\sum_{k=0}^{\ell-1} X_{i,k}^{r_i,0}\otimes X_{i+k-\ell+1,\ell-1-k}^{0,r_i}.
\end{equation}

We can check directly that the element \Cref{eq:d1} is invariant for the $A$-bimodule structure on $A\otimes A$: repeated use of the multiplication formulas for the elements $X_{\cdot,\cdot}^{\cdot,\cdot}$ given in \cite[Proposition 2.11]{MR4834012} shows that right-multiplying \Cref{eq:d1} by $X_{i,j}^{r,s}$ will annihilate all summands corresponding to
\begin{itemize}[wide]
\item indices $i'\ne i$ in the outer sum;
\item and $r_i\ne r$ in the middle summation. 
\end{itemize}
It will also annihilate the terms with
\begin{equation*}
  \ell-1-k+j\ge \ell\Longleftrightarrow k < j,
\end{equation*}
so that only the sum
\begin{equation*}
  \Delta(1)X_{i,j}^{r,s}  = \sum_{k=j}^{l-1} X_{i,k}^{r,0}\otimes X_{i+k-\ell+1,\ell-1-k+j}^{0,s} 
\end{equation*}
remains. Similarly, left multiplication by the same element $X_{i,j}^{r,s}$ will annihilate
\begin{itemize}[wide]
\item the terms indexed by $i'\ne i+j$ in the outer sum;
\item or $r_i\ne s$ in the middle one. 
\end{itemize}
This leaves
\begin{equation*}
  X_{i,j}^{r,s}\Delta(1) = \sum_{k'=0}^{\ell-1-j} X_{i,k'+j}^{r,0}\otimes X_{i+k'+j-\ell+1,\ell-1-k'}^{0,s}. 
\end{equation*}
Plainly, these two sums coincide upon relabeling $k\leftrightarrow k'+j$. 

\subsection{A family of comultiplications}\label{subse:fam}

We revisit \Cref{th:ncounit} to emphasize that there is quite a bit of freedom in producing comultiplications on self-injective algebras. Before stating the result, note that for a basic (hence Frobenius) self-injective algebra \Cref{eq:lbd} we can be more specific on what the comultiplication
\begin{equation*}
  \Delta:\Lambda\to \Lambda\otimes\Lambda
\end{equation*}
and the counit $\varepsilon:\Lambda\to \Bbbk$ look like. Write $\nu$ for the {\it Nakayama permutation} of $\Lambda$ (as in \cite[\S 1]{sy-gen}):
\begin{equation}\label{eq:nperm}
  \text{\emph{socle}}
  \ 
  \xlongequal[\quad]{\quad\text{\cite[\S I.7, Example 2]{stenstr_quot}}\quad}:
  \ 
  \mathrm{soc}(P_i)\cong \mathrm{top}(P_{\nu(i)})
  \ 
  :\xlongequal[\quad]{\quad}
  \ 
  \text{\emph{co-socle}}.
\end{equation}

\begin{remark}
  The convention \Cref{eq:nperm} matches that of \cite[p. 136]{sy-gen} (and \cite[\S 2]{MR4834012}), but {\it not} that of \cite[p. 377]{sy-bk1}, where the Nakayama permutation is the inverse of our $\nu$.
\end{remark}

A characterization of the Nakayama permutation alternative to \Cref{eq:nperm}, is as follows. Denote by $e_i\in \Lambda$ the primitive idempotents that respectively cut out the summands $P_i$: $e_i\Lambda\cong P_i$. We then have
\begin{equation}\label{eq:lr}
  e_i\Lambda \cong (\Lambda e_{\nu(i)})^*,
\end{equation}
where the right-hand side denotes the vector-space dual of the left module in question (or what would have been denoted by $D(\Lambda e_{\nu(i)})$ in \cite[\S I.8, discussion preceding Proposition 8.16]{sy-bk1}). That \Cref{eq:lr} repackages \Cref{eq:nperm} follows easily from \cite[\S I.8, Lemma 8.22]{sy-bk1}.

Then, the element $y\in \Lambda\otimes\Lambda$ inducing $\Delta$ by $\Delta(a)=ay$ will be of the form
\begin{equation}\label{eq:yspec}
  y = \sum_{i,j}\sum_{\bullet} \varphi_{\bullet, j\leftarrow i}\otimes \psi_{\bullet, \nu^{-1}(i)\leftarrow j},
\end{equation}
while the counit sends the ``maximally degenerate'' morphisms
\begin{equation}\label{eq:through}
  P_{i}\to \mathrm{top}(P_{i})\cong \mathrm{soc}(P_{\nu^{-1}i})\subseteq P_{\nu^{-1}i}
\end{equation}
to non-zero scalars and (roughly speaking) annihilates morphisms in $\Lambda_{j\leftarrow i}$ not of this form. We will formalize this in \Cref{th:topp}. A piece of terminology that will come in handy, borrowed from \cite[\S 1.3, Theorem 4]{cmz}:

\begin{definition}\label{def:fropair}
  For a Frobenius algebra $\Lambda$, a {\it Frobenius pair} $(\varepsilon,y)$ consists of
  \begin{itemize}[wide]
  \item an element $y\in \Lambda\otimes\Lambda$ invariant under the bimodule action;
  \item and the counit $\varepsilon\in \Lambda^*$ of the comultiplication
    \begin{equation}\label{eq:delind}
      \Lambda\ni a\stackrel{\Delta}{\longrightarrow} ay=ya\in \Lambda\otimes\Lambda
    \end{equation}
    induced by $y$. 
  \end{itemize}
  We will also occasionally refer to the individual components of a Frobenius pair: $\varepsilon$ is the {\it (Frobenius) counit} and $y$ is the {\it Frobenius element (or tensor)}.
\end{definition}

\begin{remark}
  The items $(\varepsilon,y)$ that constitute a Frobenius pair are also the $E$ and respectively $\beta$ of \cite[\S 27.1]{bw}.
\end{remark}

In the following statement, the notation introduced in \Cref{def:candec} is in use. We also need

\begin{definition}\label{def:maxdeg}
  A non-zero module morphism is {\it maximally degenerate} if it factors through a simple module.

  For modules $M$ and $N$ over an algebra $A$ we denote by $\mathrm{Hom}_{A}^{0}(M,N)$ the linear span of maximally degenerate morphisms $M\to N$. Its elements (which we refer to as {\it small}) are precisely those morphisms $M\to N$ that factor through the socle of $N$ (and dually, the co-socle of $M$).
\end{definition}

In preparation for stating \Cref{th:topp}, consider a Frobenius algebra $\Lambda$ with a canonical decomposition \Cref{eq:candec}. For every $1\le i\le n$ we have
\begin{equation*}
  e_{\nu^{-1}i}\Lambda e_i\cong \mathrm{Hom}_{\Lambda}(e_{\nu^{-1}i}\Lambda,e_i\Lambda),
\end{equation*}
and the corresponding small morphism space
\begin{equation}\label{eq:ee0}
  (e_{\nu^{-1}i}\Lambda e_i)^0 := \mathrm{Hom}^0_{\Lambda}(e_{\nu^{-1}i}\Lambda,e_i\Lambda),
\end{equation}
can be identified with
\begin{equation}\label{eq:sml}
  \mathrm{Hom}_{\Lambda}\left(\mathrm{top}(P_i)^{m(i)},\ \mathrm{soc}(P_{\nu^{-1}i})^{m(\nu^{-1}i)}\right).
\end{equation}
Since we are here assuming $\Lambda$ to be Frobenius, we have
\begin{itemize}[wide]
\item $m(i)=m(\nu^{-1}i)$ \cite[\S IV.6, Theorem 6.2]{sy-bk1};
\item and the two simple modules $\mathrm{top}(P_i)$ and $\mathrm{soc}(P_{\nu^{-1}i})$ featuring in \Cref{eq:sml} are isomorphic.  
\end{itemize}
Writing $S_i:=\mathrm{top}(P_i)$, \Cref{eq:sml} is
\begin{equation*}
  \mathrm{End}_{\Lambda}(S_i^{m(i)})\cong M_{m(i)}(D_i)\cong D_i\otimes M_{m(i)},
\end{equation*}
where $D_i$ is the division ring
\begin{equation*}
  D_i:=\mathrm{End}_{\Lambda}(S_i)
\end{equation*}
and $M_r$ denotes the $r\times r$ matrix algebra over $\Bbbk$.

In particular, \Cref{eq:ee0} is itself Frobenius (indeed, semisimple), so it makes sense to talk about {\it its} Frobenius elements, counits, pairs, etc.

\begin{theorem}\label{th:topp}
  Let $(e_{ij})$ be a canonical decomposition for a Frobenius algebra $\Lambda$ and $\nu$ the Nakayama permutation of $A$.

  Every Frobenius pair $(\varepsilon,y)$ of $\Lambda$ satisfies the following conditions.
  \begin{enumerate}[(a)]
  \item\label{item:3} $\varepsilon$ annihilates $e_j \Lambda e_i$ for $j\ne \nu^{-1}i$, and for each $i$ its restriction to the small space
    \begin{equation*}
      (e_{\nu^{-1}i}\Lambda e_i)^0\subseteq e_{\nu^{-1}i}\Lambda e_i
    \end{equation*}
    is a Frobenius counit for that semisimple ring.
  \item\label{item:4} The element $y\in \Lambda\otimes \Lambda$ belongs to
    \begin{equation*}
      \bigoplus_{i,j}\Lambda_{j\leftarrow i}\otimes \Lambda_{\nu^{-1}i\leftarrow j}. 
    \end{equation*}
  \end{enumerate}
\end{theorem}
\begin{proof}
  We write $\Delta$ for the comultiplication induced by $x$, as in \Cref{eq:delind}. The proof consists of a number of separate observations.

  {\bf (1): Condition \Cref{item:4} follows from \Cref{item:3}.} Writing $y=y_1\otimes y_2$ in standard Sweedler notation (\cite[Notation 1.4.2]{mntg}), we have
  \begin{equation*}
    a = (\id\otimes\varepsilon)\Delta a = y_1\varepsilon(y_2a)
  \end{equation*}
  and similarly
  \begin{equation*}
    a = (\varepsilon\otimes\id)\Delta a = \varepsilon(ay_1)y_2.
  \end{equation*}
  This means that $y=y_1\otimes y_2$ is the dual-basis tensor
  \begin{equation*}
    \sum_\alpha e_\alpha\otimes e_{\alpha}^*
  \end{equation*}
  given by the non-degenerate pairing
  \begin{equation}\label{eq:frobpair}
    \Lambda \times \Lambda  \ni (a,b)\mapsto \varepsilon(ab)\in \Bbbk,
  \end{equation}
  i.e. $(e_{\alpha})_{\alpha}$ and $(e_{\alpha}^*)_{\alpha}$ are any two dual bases with respect to that pairing:
  \begin{equation*}
    \varepsilon(e^*_{\beta}e_{\alpha}) = \delta_{\beta,\alpha}. 
  \end{equation*}
  But then, choosing the $\Lambda$-basis $(e_\alpha)$ compatible with the decomposition
  \begin{equation*}
    \Lambda =\bigoplus_{i,j} \Lambda _{j\leftarrow i},
  \end{equation*}
  the condition imposed on $\varepsilon$ in \Cref{item:3} delivers \Cref{item:4}. 

  {\bf (2): constructing $(\varepsilon,y)$ meeting condition \Cref{item:3}. } Let
  \begin{equation*}
    \varepsilon:\Lambda\cong \mathrm{End}_{\Lambda}(P_1^{m(1)}\oplus\cdots\oplus P_n^{m(n)})\to \Bbbk
  \end{equation*}
  be {\it any} linear map satisfying \Cref{item:3}; we will then show that the resulting pairing \Cref{eq:frobpair} is non-degenerate; and as $\Lambda$ is finite-dimensional, it will be enough to prove {\it right} non-degeneracy:
  \begin{equation*}
    \varepsilon(-\cdot a)\equiv 0\Longrightarrow a=0. 
  \end{equation*}
  Consider, then, a non-zero element
  \begin{equation*}
    a\in \Lambda\cong \mathrm{End}_{\Lambda}(P_1^{m(1)}\oplus\cdots\oplus P_n^{m(n)}). 
  \end{equation*}
  It must map {\it some} $P_i$ onto some non-zero submodule of $P_j$. But then one sees easily that the (non-zero) restriction (and corestriction)
  \begin{equation*}
    a|_{P_i}:P_i\to P_j
  \end{equation*}
  can be composed further with some morphism $b:P_j\to P_{\nu^{-1}i}$ so as to produce a maximally degenerate morphism
  \begin{equation}\label{eq:mymdeg}
    P_i^{m(i)}\twoheadrightarrow P_i\to P_{\nu^{-1}i}\hookrightarrow P_{\nu^{-1}i}^{m(i)}.
  \end{equation}
  By assumption $\varepsilon$ restricted to
  \begin{equation*}
    \mathrm{Hom}^0_{\Lambda}(P_i^{m(i)},\ P_{\nu^{-1}i}^{m(i)}) = \left(e_{\nu^{-1}i}\Lambda e_i\right)^0
  \end{equation*}
  is a Frobenius counit, so it will fail to annihilate some morphism with the same kernel as \Cref{eq:mymdeg}. But then we can assume \Cref{eq:mymdeg} is not annihilated by $\varepsilon$ (by further composing it with an automorphism of $P_{\nu^{-1}i}^{m(i)}$ if necessary).

  Extending $b$ to all of $\Lambda$ by annihilating $P_{j'}$, $j'\ne j$, this gives a morphism $ba$ that decomposes as the sum of
  \begin{itemize}[wide]
  \item the term \Cref{eq:mymdeg}, not annihilated by $\varepsilon$;
  \item with other terms in various $\Lambda_{\ell\leftarrow k}$ with $\ell\ne \nu^{-1}k$. 
  \end{itemize}
  By choice, $\varepsilon$ annihilates these latter but not the single former term, so $\varepsilon(ba)\ne 0$. This finishes the proof of (right) non-degeneracy.  

  {\bf (3): All Frobenius pairs satisfy \Cref{item:3}.} On the one hand, we already know at least one does (since we have constructed one already). On the other, the Frobenius counits $\varepsilon\in \Lambda^*$ are precisely the functionals that induce non-degenerate bilinear forms via \Cref{eq:frobpair}, and hence right-module isomorphisms
  \begin{equation*}
    \Lambda_{\Lambda}\cong ({}_{\Lambda}\Lambda)^*
  \end{equation*}
  (with the subscripts indicating the sidedness of the module structure under consideration). This means in particular that any two such counits ($\varepsilon$ and $\varepsilon'$, say) are related by composing with a left-module automorphism of $\Lambda$, i.e. by right-multiplying by an invertible element therein:
  \begin{equation}\label{eq:ee}
    \varepsilon'(a) = \varepsilon(ab),\ \forall a\in \Lambda
  \end{equation}
  for some invertible $b\in \Lambda$. Right multiplication by $b$ is a left-module automorphism of
  \begin{equation*}
    \Lambda = \bigoplus_{i=1}^n \Lambda e_i,
  \end{equation*}
  so a sum of $n$ automorphisms, one of each $\Lambda e_i$. Since
  \begin{equation*}
    {}_{\Lambda}\mathrm{End}(\Lambda e_i) = e_i\Lambda e_i
  \end{equation*}
  we have
  \begin{equation*}
    b\in \bigoplus_{i=1}^n e_i\Lambda e_i.
  \end{equation*}
  It follows that right $b$-multiplication preserves every summand $e_i\Lambda e_j$ of $\Lambda$, so by \Cref{eq:ee}, if $\varepsilon$ satisfies condition \Cref{item:3} so does $\varepsilon'$.

  This concludes the proof.
\end{proof}

The desired generalization of \Cref{th:ncounit}, whose statement invokes \Cref{th:topp} implicitly (applied to the basic Frobenius algebra \Cref{eq:lbd}), reads as follows.

\begin{theorem}\label{th:ncounit-bis}
  Consider a self-injective algebra \Cref{eq:amis} where \Cref{eq:lbd} is basic Frobenius. Denote by
  \begin{equation*}
    \Lambda\otimes\Lambda\ni y
    =
    \sum_{i,j}\sum_{\bullet} \varphi_{\bullet, j\leftarrow i}\otimes \psi_{\bullet, \nu^{-1}(i)\leftarrow j}
  \end{equation*}
  an element inducing a coassociative counital comultiplication on $\Lambda$, and fix, for every $1\le i\le n$, a set
  \begin{equation*}
    \cS(i)\subseteq \{1,\cdots,m(i)\}\times  \{1,\cdots,m(\nu^{-1}i)\}.
  \end{equation*}
  The element
  \begin{equation}\label{eq:xaa}
    A\otimes A\ni x :=
    \sum_{i,j} \sum_{t=1}^{m(j)} \sum_{\bullet}
    \sum_{(s,s')\in \cS(i)}
    \varphi^{t\leftarrow s}_{\bullet, j\leftarrow i}\otimes \psi^{s'\leftarrow t}_{\bullet, \nu^{-1}(i)\leftarrow j}
  \end{equation}
  is then invariant in the $A$-bimodule $A\otimes A$, and hence induces a coassociative bimodule comultiplication on $A$.
\end{theorem}

Before proceeding to the proof, we observe how this recovers the various other constructions discussed above.

\begin{example}
  Setting all $\cS(i)$ to the singleton $\{(1,1)\}$ produces the comultiplication described in the proof of \Cref{th:ncounit}. 
\end{example}

\begin{example}\label{ex:diag}
  On the other hand, setting
  \begin{equation*}
    \cS(i)=
    \begin{cases}
      \text{all of }\{1,\cdots,m(i)\}\times  \{1,\cdots,m(\nu^{-1}i)\}&\ \text{if }m(i)\ne m(\nu^{-1}(i))\\
      \text{the diagonal of }\{1,\cdots,m(i)\} &\ \text{otherwise}
    \end{cases}
  \end{equation*}
  recovers the comultiplication of \cite[Theorem 2.14]{MR4834012} for that paper's NSY algebras.
\end{example}

\pf{th:ncounit-bis}
\begin{th:ncounit-bis}
  It will be enough to show that $x$ is invariant under multiplication by any element of $A^{t\leftarrow u}_{j\leftarrow k}$ for fixed
  \begin{equation*}
    1\le j,k\le n,\quad 1\le t\le m(j)\quad \text{and}\quad 1\le u\le m(k);
  \end{equation*}
  or:
  \begin{equation*}
    a^{t\leftarrow u}x=xa^{t\leftarrow u},\ \forall a\in \Lambda_{j\leftarrow k}\cong A^{1\leftarrow 1}_{j\leftarrow k}.
  \end{equation*}  
  To check this, in turn, it will suffice to equate, for each $i$, what we will call the {\it $i$-components} of $a^{t\leftarrow u}x$ and $xa^{t\leftarrow u}$: the partial sums thereof collecting the terms lying in
  \begin{equation*}
    \bigoplus A_{\bullet\leftarrow i}^{\bullet\leftarrow \bullet}\otimes A_{\nu^{-1}i\leftarrow \bullet}^{\bullet\leftarrow \bullet}
  \end{equation*}
  (with summation over all choices of placeholder).

  We regard $y\in \Lambda\otimes\Lambda$ as an element
  \begin{equation*}
    \sum_{i,j}\sum_{\bullet} \varphi^{1\leftarrow 1}_{\bullet, j\leftarrow i}\otimes \psi^{1\leftarrow 1}_{\bullet, \nu^{-1}(i)\leftarrow j} \in A\otimes A
  \end{equation*}
  via the usual identification
  \begin{equation*}
    \Lambda\cong \bigoplus_{i,j}A_{j\leftarrow i}^{1\leftarrow 1}.
  \end{equation*}
  Now note that the $i$-component of $xa^{t\leftarrow u}$ can be obtained by applying the morphism
  \begin{equation}\label{eq:spread}
    \sum_{(s,s')\in \cS(i)}(\bullet)^{t\leftarrow s}\otimes(\bullet)^{s'\leftarrow u}:
    \bigoplus_j A^{1\leftarrow 1}_{j\leftarrow i}\otimes A^{1\leftarrow 1}_{\nu^{-1}i\leftarrow j}
    \to
    \bigoplus_{(s,s')\in \cS(i)} \bigoplus_j A^{t\leftarrow s}_{j\leftarrow i}\otimes A^{s'\leftarrow u}_{\nu^{-1}i\leftarrow j}
  \end{equation}
  to the $i$-component of $ya$. Similarly, the $i$-component of $a^{t\leftarrow u}x$ is obtained from that of $ay$ by applying \Cref{eq:spread}. Since $ay=ya$ by assumption, we are done. 
\end{th:ncounit-bis}

The following result highlights the importance of the choice of $\cS(i)$ in \Cref{ex:diag}.

\begin{proposition}\label{pr:whencounit}
  Under the hypotheses of \Cref{th:ncounit-bis}, the comultiplication
  \begin{equation*}
    A\ni a\mapsto ax = xa\in A\otimes A
  \end{equation*}
  induced by the element $x$ of \Cref{eq:xaa} is counital if and only if for all $i$, the set $\cS(i)$ is the graph of a bijection.
\end{proposition}
\begin{proof}
  We fix a Frobenius pair $(\varepsilon,y)$ for $\Lambda$ throughout, and prove the two implications separately.

  {\bf ($\Leftarrow$)} The assumption is that each $\cS(i)$ is the graph of a bijection
  \begin{equation*}
    \varphi_i:\{1,\cdots,m(i)\}\to \{1,\cdots,m(\nu^{-1}i)\},
  \end{equation*}
  so in particular $m(i)=m(\nu^{-1}i)$ for all $i$ (precisely the necessary and sufficient condition that ensures $A$ is Frobenius \cite[\S IV.6, Theorem 6.2]{sy-bk1}). In short:
  \begin{equation*}
    \cS(i) = \{(s,\varphi(s))\ |\ 1\le s\le m(i)\},\ \forall 1\le i\le n.
  \end{equation*}
  We proceed to construct a counit $\varepsilon_A$ for $A$ (rendering $\Delta$ counital), using the original counit $\varepsilon\in\Lambda^*$ fixed above. As we already know that $x\in A\otimes A$ is invariant (i.e. $ax=xa$), all we need is to ensure that
  \begin{equation}\label{eq:allid}
    (\varepsilon_A\otimes\id)x = 1_A = (\id\otimes\varepsilon_A)x. 
  \end{equation}
  To achieve this, define $\varepsilon_A$ to be
  \begin{itemize}[wide]
  \item on each summand $A^{\varphi_i^{-1}t\leftarrow t}_{\nu^{-1}i\leftarrow i}$, the composition
    \begin{equation*}
      A^{\varphi_i^{-1}t\leftarrow t}_{\nu^{-1}i\leftarrow i}      
      \xrightarrow{\quad \left(\bullet^{\varphi_i^{-1}t\leftarrow t}\right)^{-1}\quad}
      \Lambda_{\nu^{-1}i\leftarrow i}
      \xrightarrow{\quad\varepsilon\quad}
      \Bbbk;      
    \end{equation*}
  \item and zero on other summands $A^{t\leftarrow s}_{j\leftarrow i}$.
  \end{itemize}
  Applying $(1\otimes\varepsilon_A)$ to $x$ will then have the effect of annihilating
  \begin{itemize}[wide]
  \item all right-hand tensorands of \Cref{eq:xaa} {\it not} in $A^{\varphi_i^{-1}t\leftarrow t}_{\nu^{-1}i\leftarrow i}$;
  \item and hence all terms therein whose left-hand tensorand is {\it not} in $A^{t\leftarrow t}_{i\leftarrow i}$.
  \end{itemize}
  This means that $(\id\otimes\varepsilon_A)x$ is obtained from
  \begin{equation*}
    (\id\otimes\varepsilon)y = 1_{\Lambda} = \sum_i \id_i \in \bigoplus_i \Lambda_{i\leftarrow i}
  \end{equation*}
  by applying
  \begin{equation*}
    \sum_{t=1}^{m(i)}(\bullet)^{t\leftarrow t} : \Lambda_{i\leftarrow i}\to \bigoplus A^{t\leftarrow t}_{i\leftarrow i},
  \end{equation*}
  respectively, to each $\Lambda_{i\leftarrow i}$ component $\id_i$. But this means precisely that
  \begin{equation*}
    (\id\otimes\varepsilon_A)x = \sum_i\sum_{t=1}^{m(i)}\id_i^{t\leftarrow t} = 1_{A},
  \end{equation*}
  and we are done with the right-hand half of \Cref{eq:allid}. The left-hand half is entirely analogous, so we omit the parallel argument.

  {\bf ($\Rightarrow$)} The existence of a counital comultiplication (coassociative and a bimodule morphism) means precisely that $A$ is Frobenius, which in turn means \cite[\S IV.6, Theorem 6.2]{sy-bk1} that
  \begin{equation*}
    m(i) = m(\nu^{-1}i),\ \forall 1\le i\le n.
  \end{equation*}
  It remains to argue that for every $i$, the restrictions
  \begin{equation}\label{eq:2proj}
    \begin{tikzpicture}[auto,baseline=(current  bounding  box.center)]
      \path[anchor=base] 
      (0,0) node (l) {$\cS(i)$}
      +(3,0) node (m) {$\{1,\cdots,m(i)\}^2$}
      +(8,0) node (r) {$\{1,\cdots,m(i)\}$}
      ;
      \draw[right hook->] (l) to[bend left=0] node[pos=.5,auto] {$\scriptstyle $} (m);
      \draw[->] (m) to[bend left=6] node[pos=.5,auto] {$\scriptstyle \text{left projection}$} (r);
      \draw[->] (m) to[bend right=6] node[pos=.5,auto,swap] {$\scriptstyle \text{right projection}$} (r);
    \end{tikzpicture}
  \end{equation}
  of the two Cartesian projections are both bijective.
 
  The counitality assumption means that there is some $\varepsilon\in A^*$ with
  \begin{equation*}
    (\varepsilon\otimes\id)x = 1_A = (\id\otimes\varepsilon)x.
  \end{equation*}
  This means, in particular, that an application of $\varepsilon$ to the right-hand side of \Cref{eq:xaa} produces the identity $1_A$. Since the latter decomposes as a sum of (non-zero) components in $A^{t\leftarrow t}_{i\leftarrow i}$ for
  \begin{equation*}
    1\le i\le n,\ 1\le t\le m(i),
  \end{equation*}
  it follows that the top composition in \Cref{eq:2proj} is surjective; so is the bottom composition, by a parallel argument.

  We next turn to proving the injectivity, say, of the bottom path in \Cref{eq:2proj} (again, the other case is analogous). Because \Cref{eq:xaa} simply sums, for each $i$, over all $(s,s')\in \cS(i)$, for
  \begin{equation*}
    (s,t)\ \text{and}\ (s,t')\in \cS(i)
  \end{equation*}
  and any $1\le j\le n$ the $A^{t\leftarrow t}_{j\leftarrow j}$ and $A^{t'\leftarrow t}_{j\leftarrow j}$-components of $(\varepsilon\otimes\id)x$ are mutual images through
  \begin{equation*}
    A^{t\leftarrow t}_{j\leftarrow j}
    \xrightarrow{\quad \left(\bullet^{t\leftarrow t}\right)^{-1}\quad}
    \Lambda_{j\leftarrow j}
    \xrightarrow{\quad \bullet^{t'\leftarrow t}\quad}
    A^{t'\leftarrow t}_{j\leftarrow j}
  \end{equation*}
  and its inverse. Since the former component is non-zero so is the latter, but this cannot be unless $t'=t$. This proves the desired injectivity of
  \begin{equation*}
    \cS(i=\nu j)\ni (s,t)\mapsto s,
  \end{equation*}
  finishing the proof. 
\end{proof}



\addcontentsline{toc}{section}{References}

\Addresses

\end{document}